\documentclass{amsart}

\newtheorem{theorem}{Theorem}%[section]
\newtheorem{lemma}[theorem]{Lemma}
\newtheorem{proposition}[theorem]{Proposition}
\newtheorem{corollary}[theorem]{Corollary}

\theoremstyle{definition}

\newtheorem{example}[theorem]{Example}

\usepackage{amscd,amssymb}
\begin{document}

\title[Inner and Outer Automorphisms]
{Inner and Outer Automorphisms\\
of Free Metabelian Nilpotent Lie algebras}

\author[Vesselin Drensky and \c{S}ehmus F\i nd\i k]
{Vesselin Drensky and \c{S}ehmus F\i nd\i k}
\address{Institute of Mathematics and Informatics,
Bulgarian Academy of Sciences,
1113 Sofia, Bulgaria}
\email{drensky@math.bas.bg}
\address{Department of Mathematics,
\c{C}ukurova University, 01330 Balcal\i,
 Adana, Turkey}
\email{sfindik@cu.edu.tr}

\thanks
{The research of the second named author was partially supported by the
 Council of Higher Education (Y\"OK) in Turkey}


\begin{abstract} Let $L_{m,c}$ be the free metabelian nilpotent of class $c$ Lie algebra of rank $m$
over a field $K$ of characteristic 0. We describe the groups of inner and outer automorphisms of $L_{m,c}$.
To obtain this result we first describe the groups of inner and continuous outer automorphisms of the completion
$\widehat{F_m}$ with respect to the formal power series topology of the free metabelian Lie algebra $F_m$ of rank $m$.
\end{abstract}

\maketitle

\section*{Introduction}

Let $L_m$ be the free Lie algebra of rank $m\geq 2$ over a field $K$ of characteristic 0 with free generators
$x_1,\ldots,x_m$ and let $L_{m,c}=L_m/(L_m''+L_m^{c+1})$ be the free metabelian nilpotent of class $c$
Lie algebra. This is the relatively free algebra of rank $m$ in the variety of Lie
algebras $\mathfrak A^2\cap \mathfrak N_c$, where $\mathfrak A^2$ is the
metabelian (solvable of class 2) variety of Lie algebras and ${\mathfrak N}_c$ is the variety of
all nilpotent Lie algebras of class at most $c$.

The initial goal of our paper was to describe the groups of inner automorphisms $\text{\rm Inn}(L_{m,c})$
and outer automorphisms $\text{\rm Out}(L_{m,c})$  of
the Lie algebra $L_{m,c}$. The automorphism group $\text{\rm Aut}(L_{m,c})$ is a semidirect product of the normal subgroup
$\text{\rm IA}(L_{m,c})$ of the automorphisms which induce the identity map modulo the commutator ideal
of $L_{m,c}$ and the general linear group $\text{\rm GL}_m(K)$.
Since $\text{\rm Inn}(L_{m,c})\subset \text{\rm IA}(L_{m,c})$,
for the description of the factor group
$\text{\rm Out}(L_{m,c})=\text{\rm Aut}(L_{m,c})/\text{\rm Inn}(L_{m,c})$ it is sufficient to know only
$\text{\rm IA}(L_{m,c})/\text{\rm Inn}(L_{m,c})$.

The composition of two inner automorphisms is an inner automorphism obtained by
the Baker-Campbell-Hausdorff formula which gives the solution $z$ of the equation $e^z=e^x\cdot e^y$
for non-commuting $x$ and $y$.
If $x,y$ are the generators of the free associative algebra $A=K\langle x,y\rangle$, then
$z$ is a formal power series in the completion $\widehat{A}$ with respect to the formal power series topology.
The homogeneous components of $z$ are Lie elements.
Hence, it is natural to work in the completion $\widehat{F_m}$ with respect to the formal power series topology
of the free metabelian Lie algebra $F_m=L_m/L''_m$ and to study the groups of its inner automorphisms and
of its continuous outer automorphisms. Then the results for $\text{\rm Inn}(L_{m,c})$ and $\text{\rm Out}(L_{m,c})$
follow easily by factorization modulo $(\widehat{F_m})^{c+1}$.
Gerritzen \cite{G} found a simple version of the Baker-Campbell-Hausdorff
formula when applied to the algebra $\widehat{F_2}$. This provides a nice expression of
the composition of inner automorphisms
in $\widehat{F_m}$.

A result of Shmel'kin \cite{Sh} states that the free metabelian Lie algebra $F_m$ can be embedded into the abelian
wreath product $A_m\text{\rm wr}B_m$, where $A_m$ and $B_m$ are $m$-dimensional abelian Lie algebras
with bases  $\{a_1,\ldots,a_m\}$ and $\{b_1,\ldots,b_m\}$, respectively.
The elements of $A_m\text{\rm wr}B_m$ are of the form $\sum_{i=1}^ma_if_i(t_1,\ldots,t_m)+\sum_{i=1}^m\beta_ib_i$,
where $f_i$ are polynomials in $K[t_1,\ldots,t_m]$ and $\beta_i\in K$. This allows to introduce partial derivatives
in $F_m$ with values in $K[t_1,\ldots,t_m]$ and the Jacobian matrix $J(\phi)$ of an endomorphism $\phi$ of $F_m$.
Restricted on the semigroup $\text{\rm IE}(F_m)$ of endomorphisms of $F_m$ which are identical modulo
the commutator ideal $F_m'$, the map $J:\phi\to J(\phi)$ is a semigroup monomorphism of $\text{\rm IE}(F_m)$
into the multiplicative semigroup of the algebra $M_m(K[t_1,\ldots,t_m])$ of $m\times m$ matrices
with entries from $K[t_1,\ldots,t_m]$.
We give the explicit form of the Jacobian matrices of inner automorphisms of $\widehat{F_m}$ and of
the coset representatives of the continuous outer automorphisms in $\text{\rm IOut}(\widehat{F_m})$.

Finally we transfer the obtained results to the algebra $L_{m,c}$ and obtain the description of $\text{\rm Inn}(L_{m,c})$
and $\text{\rm IOut}(L_{m,c})$.

\section{Preliminaries}

Let $L_m$ be the free Lie algebra of rank $m\geq 2$ over a field $K$ of characteristic 0
with free generators $x_1,\ldots,x_m$ and let $L_{m,c}=L_m/(L_m''+L_m^{c+1})$ be the free
metabelian nilpotent of class $c$ Lie algebra. It is freely generated by
$y_1,\ldots,y_m$, where $y_i=x_i+(L_m''+L_m^{c+1})$, $i=1,\ldots,m$.
We use the commutator notation for the Lie multiplication. Our commutators are left normed:
\[
[u_1,\ldots,u_{n-1},u_n]=[[u_1,\ldots,u_{n-1}],u_n],\quad n=3,4,\ldots.
\]
In particular,
\[
L_m^k=\underbrace{[L_m,\ldots,L_m]}_{k\text{ times}}.
\]
For each $v\in L_{m,c}$, the linear operator $\text{\rm ad}v:L_{m,c}\to L_{m,c}$ defined by
\[
u(\text{\rm ad}v)=[u,v],\quad u\in L_{m,c},
\]
is a derivation of $L_{m,c}$ which is nilpotent and $\text{\rm ad}^cv=0$
because $L_{m,c}^{c+1}=0$.
Hence the linear operator
\[
\exp(\text{\rm ad}v)=1+\frac{\text{\rm ad}v}{1!}+\frac{\text{\rm ad}^2v}{2!}+\cdots
=1+\frac{\text{\rm ad}v}{1!}+\frac{\text{\rm ad}^2v}{2!}+\cdots+\frac{\text{\rm ad}^{c-1}v}{(c-1)!}
\]
is well defined and is an automorphism  of $L_{m,c}$.
The set of all such automorphisms forms a normal subgroup of the group of all
automorphisms $\text{\rm Aut}(L_{m,c})$ of $L_{m,c}$.
This group is called the inner automorphism group of $L_{m,c}$ and is denoted by $\text{\rm Inn}(L_{m,c})$.
The factor group $\text{\rm Aut}(L_{m,c})/\text{\rm Inn}(L_{m,c})$ is called the outer automorphism group of
$L_{m,c}$ and is denoted by $\text{\rm Out}(L_{m,c})$. The automorphism group $\text{\rm Aut}(L_{m,c})$
is a semidirect product of the normal subgroup
$\text{\rm IA}(L_{m,c})$ of the automorphisms which induce the identity map modulo the
commutator ideal of $L_{m,c}$ and the general
linear group $\text{\rm GL}_m(K)$. Since $\text{\rm Inn}(L_{m,c})\subset \text{\rm IA}(L_{m,c})$,
for the description of
$\text{\rm Out}(L_{m,c})=\text{\rm Aut}(L_{m,c})/\text{\rm Inn}(L_{m,c})$ it is sufficient to know only
$\text{\rm IA}(L_{m,c})/\text{\rm Inn}(L_{m,c})$.

Let $R$ be an algebra over a field $K$ of characteristic 0.
We consider the topology on $R$ induced by the series
$R\supseteq R^2\supseteq R^3\supseteq\cdots$. This is the topology
in which the sets $r+R^k$, $r\in R$, $k\geq 1$, form a basis for the open sets. It is called the
{\it formal power series topology} on $R$.

Let $F_m=L_m/L''_m$ be the free metabelian Lie algebra
of rank $m$. We shall denote the free generators of $F_m$ with the same symbols $y_1,\ldots,y_m$ as the free generators of
$L_{m,c}$, but now $y_i=x_i+L''_m$, $i=1,\ldots,m$. It is well known, see e.g. \cite{Ba}, that
\[
[y_{i_1},y_{i_2},y_{i_{\sigma(i_3)}},\ldots,y_{i_{\sigma(k)}}]
=[y_{i_1},y_{i_2},y_{i_3},\ldots,y_{i_k}],
\]
where $\sigma$ is an arbitrary permutation of $3,\ldots,k$ and that $F_m'$ has a basis consisting of all
\[
[y_{i_1},y_{i_2},y_{i_3},\ldots,y_{i_k}],\quad 1\leq i_j\leq m,\quad i_1>i_2\leq i_3\leq\cdots\leq i_k.
\]
Let $(F_m)_{(k)}$ be the subspace of $F_m$ spanned by all monomials
of total degree $k$ in $y_1,\ldots,y_m$.
We consider the formal power series topology on $F_m$. The completion
$\widehat{F_m}$ of $F_m$ with respect to this topology
may be identified with the complete direct sum $\widehat \bigoplus_{i\geq 1}(F_m)_{(i)}$.
The elements $f\in \widehat{F_m}$ are formal power series
\[
f=f_1+f_2+\cdots,\quad f_i\in (F_m)_{(i)},\quad i=1,2,\ldots.
\]
The composition of two inner automorphisms in $\text{\rm Inn}(L_{m,c})$ is also an inner automorphism.
It can be obtained by
the Baker-Campbell-Hausdorff formula which gives the solution $z$ of the equation $e^z=e^x\cdot e^y$
for non-commuting $x$ and $y$, see e.g. \cite{Bo} and \cite{Se}.
If $x,y$ are the generators of the free associative algebra $A=K\langle x,y\rangle$, then
\[
z=x+y+\frac{[x,y]}{2}-\frac{[x,y,x]}{12}+\frac{[x,y,y]}{12}-\frac{[x,y,x,y]}{24}+\cdots
\]
is a formal power series in the completion $\widehat{A}$ with respect to the formal power series topology.
The homogeneous components of $z$ are Lie elements.
Hence, studying the inner and the outer automorphisms of $L_{m,c}$,
it is convenient to work in the completion $\widehat{F_m}$ and to study the groups of its inner automorphisms and
of its continuous outer automorphisms. Working in the algebra $F_2=L_2/L_2''$
with generators $y_1=x$, $y_2=y$, the element $z\in\widehat{F_2}$ in the
Baker-Campbell-Hausdorff formula has the form
\[
z=x+y+\sum _{a,b\geq 0}c_{ab}[x,y,\underbrace{x,\ldots,x}_{a},\underbrace{y,\ldots,y}_{b}]
=x+y+\sum _{a,b\geq 0}c_{ab}[x,y]\text{\rm ad}^ax\text{\rm ad}^by
\]
\[
=x+y+[x,y]c(\text{\rm ad}x,\text{\rm ad}y),
\]
where
\[
c(t,u)=\sum_{a,b\geq 0}c_{ab}t^au^b\in{\mathbb Q}[[t,u]].
\]
Gerritzen \cite{G} found a nice description of the formal power series in commuting variables $c(t,u)$
which corresponds to $z\in\widehat{F_2}$. Further
we shall use it to obtain an expression of the composition of inner automorphisms
in $\widehat{F_m}$. We present a slightly modified version of the result of Gerritzen due to the fact
that he uses right normed commutators (and not left normed commutators as we do).
Recall that $F_2$ is isomorphic to the Lie subalgebra generated by $x$ and $y$
in the associative algebra $A/C^2$, where $C=A[A,A]A$ is the commutator ideal of $A=K\langle x,y\rangle$.

\begin{proposition}\label{formula of Gerritzen}{\rm(Gerritzen \cite{G})}
If $z\in\widehat{F_2}$ is the solution of the equation $e^z=e^x\cdot e^y$
in the algebra $\widehat{A/C^2}$, then
\[
z=x+y+[x,y]c(\text{\rm ad}x,\text{\rm ad}y),
\]
where
\[
c(t,u)=\frac{e^uh(t)-h(u)}{e^{t+u}-1}\in K[[t,u]],\quad h(v)=\frac{e^v-1}{v}\in K[[v]].
\]
\end{proposition}

\begin{proof}
Following Gerritzen \cite{G}, if $D_v$ is the derivation of the free metabelian algebra $F_2$
defined by $D_v(u)=[v,u]$, $u\in F_2$, then the solution $z\in \widehat{F_2}$ of the equation $e^z=e^x\cdot e^y$
has the form
\[
z=x+y+H_0(D_x,D_y)[x,y],
\]
where
\[
H_0(t,u)=\frac{h(-t)-h(u)}{e^{-t}-e^u}, \quad h(v)=\frac{e^v-1}{v}.
\]
Since $D_x=-\text{\rm ad}x$ and $D_y=-\text{\rm ad}y$, we obtain
\[
z=x+y+H_0(D_x,D_y)[x,y]=x+y+[x,y]H_0(-\text{\rm ad}x,-\text{\rm ad}y)
=x+y+[x,y]c(\text{\rm ad}x,\text{\rm ad}y),
\]
\[
c(t,u)=H_0(-t,-u)=\frac{h(t)-h(-u)}{e^t-e^{-u}}
=\frac{e^uh(t)-h(u)}{e^{t+u}-1}
\]
after easy computations.
\end{proof}

Pay attention that both the nominator and the denominator of $c(t,u)$ are divisible by
$t+u$. After the cancellation of the common factor $t+u$ the denominator becomes an invertible element of $K[[t,u]]$.

Now we collect the necessary information about wreath products, Jacobian matrices
and formal power series. For details and references see e.g. \cite{BD}.
Let $K[t_1,\ldots,t_m]$ be the
(commutative) polynomial algebra over $K$ freely generated by the
variables $t_1,\ldots,t_m$ and let $A_m$ and $B_m$ be the abelian
Lie algebras with bases $\{a_1,\ldots,a_m\}$ and
$\{b_1,\ldots,b_m\}$, respectively. Let $C_m$ be the free right
$K[t_1,\ldots,t_m]$-module with free generators $a_1,\ldots,a_m$.
We give it the structure of a Lie algebra with trivial multiplication.
The abelian wreath product
$A_m\text{\rm wr}B_m$ is equal to the semidirect sum $C_m\leftthreetimes B_m$. The elements
of $A_m\text{\rm wr}B_m$ are of the form
$\sum_{i=1}^ma_if_i(t_1,\ldots,t_m)+\sum_{i=1}^m\beta_ib_i$, where
$f_i$ are polynomials in $K[t_1,\ldots,t_m]$ and $\beta_i\in K$.
The multiplication in $A_m\text{\rm wr}B_m$ is defined by
\[
[C_m,C_m]=[B_m,B_m]=0,
\]
\[
[a_if_i(t_1,\ldots,t_m),b_j]=a_if_i(t_1,\ldots,t_m)t_j,\quad i,j=1,\ldots,m.
\]
Hence $A_m\text{\rm wr}B_m$ is a metabelian Lie algebra and every mapping $\{y_1,\ldots,y_m\}\to A_m\text{\rm wr}B_m$
can be extended to a homomorphism $F_m\to A_m\text{\rm wr}B_m$.
As a special case of the embedding theorem of Shmel'kin \cite{Sh},
the homomorphism $\varepsilon:F_m\to A_m\text{\rm wr}B_m$ defined by
$\varepsilon(y_i)=a_i+b_i$, $i=1,\ldots,m$, is a monomorphism. If
\[
f=\sum[y_i,y_j]f_{ij}(\text{\rm ad}y_1,\ldots,\text{\rm ad}y_m),\quad f_{ij}(t_1,\ldots,t_m)\in K[t_1,\ldots,t_m],
\]
then
\[
\varepsilon(f)=\sum(a_it_j-a_jt_i)f_{ij}(t_1,\ldots,t_m).
\]
The next lemma follows from \cite{Sh}, see also \cite {BD}.

\begin{lemma}\label{metabelian rule}
The element $\sum_{i=1}^ma_if_i(t_1,\ldots,t_m)$ of
$C_m$ belongs to $\varepsilon (F_m')$ if and only if
$\sum_{i=1}^mt_if_i(t_1,\ldots,t_m)=0$.
\end{lemma}

The embedding of $F_m$ into $A_m\text{\rm wr}B_m$ allows to introduce partial derivatives
in $F_m$ with values in $K[t_1,\ldots,t_m]$. If $f\in F_m$ and
\[
\varepsilon(f)=\sum_{i=1}^m\beta_ib_i+\sum_{i=1}^ma_if_i(t_1,\ldots,t_m),\quad \beta_i\in K,f_i\in K[t_1,\ldots,t_m],
\]
then
\[
\frac{\partial f}{\partial y_i}=f_i(t_1,\ldots,t_m).
\]
The Jacobian matrix $J(\phi)$ of an endomorphism $\phi$ of $F_m$
is defined as
\[
J(\phi)=\left(\frac {\partial \phi({y_j})}{\partial y_i}\right)
=\left(\begin{matrix}
\frac {\partial\phi({y_1})}{\partial y_1}&\cdots&\frac {\partial \phi({y_m})}{\partial y_1}\\
\vdots&\ddots&\vdots\\
\frac {\partial\phi({y_1})}{\partial y_m}&\cdots&\frac {\partial \phi({y_m})}{\partial y_m}\\
\end{matrix}\right)\in M_m(K[t_1,\ldots,t_m]),
\]
where $M_m(K[t_1,\ldots,t_m])$ is the associative algebra of $m\times m$ matrices with entries from
$K[t_1,\ldots,t_m]$. Let $\text{\rm IE}(F_m)$ be the multiplicative semigroup of all endomorphisms
of $F_m$ which are identical modulo the commutator ideal $F_m'$.
Let $I_m$ be the identity $m\times m$ matrix and let $S$ be the subspace of
$M_m(K[t_1,\ldots,t_m])$ defined by
\[
S=\left \{(f_{ij})\in M_m(K[t_1,\ldots,t_m]) \mid \sum_{i=1}^mt_if_{ij}=0,j=1,\ldots,m\right \}.
\]
Clearly $I_m+S$ is a subsemigroup of the multiplicative group of $M_m(K[t_1,\ldots,t_m])$.
The completion with respect to the formal power series topology of $A_m\text{\rm wr}B_m$
is the semidirect sum $\widehat{C_m}\leftthreetimes B_m$, where
\[
\widehat{C_m}=\bigoplus_{i=1}^ma_iK[[t_1,\ldots,t_m]].
\]
If $\phi\in \text{\rm IE}(\widehat{F_m})$, then $J(\phi)=I_m+(s_{ij})$,
where $s_{ij}\in K[[t_1,\ldots,t_m]]$.
It is easy to check that if $\phi,\psi \in \text{\rm IE}(\widehat{F_m})$ then $J(\phi \psi)=J(\phi)J(\psi)$.
The following proposition is well known, see e.g. \cite {BD}.

\begin{proposition}\label{met}
The map $J:\text{\rm IE}(F_m)\to I_m+S$ defined by
$\phi\to J(\phi)$ is an isomorphism of the semigroups $\text{\rm IE}(F_m)$ and $I_m+S$.
It extends to an isomorphism
between the group $\text{\rm IE}(\widehat{F_m})=\text{\rm IA}(\widehat{F_m})$ of continuous {\rm IA}- automorphisms
of $\widehat{F_m}$ and the multiplicative group $I_m+\widehat{S}$.
\end{proposition}

\section{Main Results}

In this section we give a formula for the multiplication of inner automorphisms of $\widehat{F_m}$. We
also find the explicit form of the Jacobian matrix of the inner automorphisms of $\widehat{F_m}$ and of
the coset representatives of the continuous outer automorphisms in $\text{\rm IOut}(\widehat{F_m})$. Finally we
transfer the obtained results to the algebra $L_{m,c}$ and obtain the description of $\text{\rm Inn}(L_{m,c})$
and $\text{\rm IOut}(L_{m,c})$.

Let $\text{\rm Inn}(\widehat{F_m})$ denote the set of all inner automorphisms of $\widehat{F_m}$ which are of the form
$\exp(\text{\rm ad}u)$, $u\in \widehat{F_m}$. Our first goal is to give a multiplication rule for
$\text{\rm Inn}(\widehat{F_m})$. For $u\in\widehat{F_m}$ we fix the notation $u=\overline{u}+u_0$,
where $\overline{u}$ is the linear component of $u$ and $u_0\in \widehat{F_m'}$.

\begin{theorem}\label{multiplication in Inn}
Let $u,v\in \widehat{F_m}$. Then the solution $w=w(u,v)\in \widehat{F_m}$ of the equation
$\exp(\text{\rm ad}u)\cdot \exp(\text{\rm ad}v)=\exp(\text{\rm ad}w)$ is
\[
w=H(\overline{u},\overline{v})+u_0(1+(\text{\rm ad}\overline{v})c(\text{\rm ad}\overline{u},\text{\rm ad}\overline{v}))
+v_0(1-(\text{\rm ad}\overline{u})c(\text{\rm ad}\overline{u},\text{\rm ad}\overline{v})),
\]
where
\[
H(\overline{u},\overline{v})=\overline{u}+\overline{v}+[\overline{u},\overline{v}]c(\text{\rm ad}\overline{u},\text{\rm ad}\overline{v}),
\]
\[
c(t_1,t_2)=\frac{e^{t_2}h(t_1)-h(t_2)}{e^{t_1+t_2}-1},\quad h(t)=\frac{e^t-1}{t}.
\]
\end{theorem}

\begin{proof}
The adjoint operator $\text{\rm ad}:\widehat{F_m}\to\text{\rm End}_K\widehat{F_m}$ is a Lie algebra homomorphism.
The algebra $\text{\rm ad}(\widehat{F_m})$ is metabelian, as a homomorphic image of $\widehat{F_m}$.
Hence we may apply the formula of Gerritzen from Proposition \ref{formula of Gerritzen}
for $x=\text{\rm ad}u$, $y=\text{\rm ad}v$, $z=\text{\rm ad}w$. Therefore
\[
\text{\rm ad}w=\text{\rm ad}u+\text{\rm ad}v+[\text{\rm ad}u,\text{\rm ad}v]c(\text{\rm ad}(\text{\rm ad}u),\text{\rm ad}(\text{\rm ad}v))
\]
\[
=\text{\rm ad}u+\text{\rm ad}v+\text{\rm ad}([u,v]c(\text{\rm ad}u,\text{\rm ad}v)).
\]
Since the algebra $\widehat{F_m}$ has no centre, the adjoint representation is faithful and
\[
w=u+v+[u,v]c(\text{\rm ad}u,\text{\rm ad}v).
\]
The metabelian low gives that
\[
[u,v]c(\text{\rm ad}u,\text{\rm ad}v)=[\overline{u}+u_0,\overline{v}+v_0]c(\text{\rm ad}(\overline{u}+u_0),\text{\rm ad}(\overline{v}+v_0))
\]
\[
=([\overline{u},\overline{v}]+[u_0,\overline{v}]-[v_0,\overline{u}])c(\text{\rm ad}\overline{u},\text{\rm ad}\overline{v}),
\]
\[
w=(\overline{u}+u_0)+(\overline{v}+v_0)
+([\overline{u},\overline{v}]+[u_0,\overline{v}]-[v_0,\overline{u}])c(\text{\rm ad}\overline{u},\text{\rm ad}\overline{v})
\]
\[
=H(\overline{u},\overline{v})+u_0(1+(\text{\rm ad}\overline{v})c(\text{\rm ad}\overline{u},\text{\rm ad}\overline{v}))
+v_0(1-(\text{\rm ad}\overline{u})c(\text{\rm ad}\overline{u},\text{\rm ad}\overline{v})).
\]
\end{proof}

Our next objective is to give the explicit form of the Jacobian matrix of the inner automorphisms of $\widehat{F_m}$.

\begin{theorem}\label{Jacobian of inner autos}
Let $u=\overline{u}+u_0\in \widehat{F_m}$, where
\[
\overline{u}=\sum_{r=1}^mc_ry_r, \quad c_r\in K, r=1,\ldots,m,
\]
is the linear component of $u$ and
\[
u_0=\sum_{p>q} [y_p,y_q]h_{pq}(\text{\rm ad}y_q,\ldots,\text{\rm ad}y_m).
\]
Then
\[
J(\exp(\text{\rm ad}u))=J(\exp(\text{\rm ad}\overline{u}))+D_0T=I_m+(\overline{D}+D_0)T,
\]
\[
\overline{D}=\left(\frac{\partial[y_j,\overline{u}]}{\partial y_i}\right),
\quad D_0=\left(\frac{\partial[y_j,u_0]}{\partial y_i}\right),
\]
\[
T=\sum_{k\geq 0}\frac{t^k}{(k+1)!},\quad t=\sum_{r=1}^mc_rt_r.
\]
More precisely
\[
\overline{D}=\left(
\begin{array}{cccc}
c_2t_2+c_3t_3+\cdots+c_mt_m&-c_1t_2&\cdots&-c_1t_m\\
-c_2t_1&c_1t_1+c_3t_3+\cdots+c_mt_m&\cdots&-c_2t_m\\
\vdots&\vdots&\ddots&\vdots\\
-c_mt_1&-c_mt_2&\cdots&\sum_{k=1}^{m-1}c_kt_k
\end{array}
\right),
\]
\[
D_0=\left(
\begin{array}{cccc}
-t_1f_1&-t_2f_1&\cdots&-t_mf_1\\
-t_1 f_2&-t_2f_2&\cdots&-t_mf_2\\
\vdots&\vdots&\ddots&\vdots\\
-t_1f_m&-t_2f_m&\cdots&-t_mf_m
\end{array}
\right),
\]
where
\[
f_i=\sum_{p>q} \frac{\partial\left([y_p,y_q]h_{pq}(\text{\rm ad}y_q,\ldots,\text{\rm ad}y_m)\right)}{\partial y_i}
\]
\[
=\sum_{q=1}^{i-1}t_qh_{iq}(t_q,\ldots,t_m)-\sum_{p=i+1}^mt_ph_{pi}(t_i,\ldots,t_m).
\]
\end{theorem}

\begin{proof}
If $u=\overline{u}+u_0$, then
\[
\exp(\text{\rm ad}u)(y_j)=y_i+[y_j,u]\sum_{k\geq 0}\frac{(\text{\rm ad}\overline{u})^k}{(k+1)!},\quad j=1,\ldots,m,
\]
because $u_0\in\widehat{F_m}$ and $\text{\rm ad}u_0$ acts trivially on the commutator ideal of $\widehat{F_m}$. Hence
\[
\exp(\text{\rm ad}u)(y_j)=y_i+[y_j,\overline{u}+u_0]\sum_{k\geq 0}\frac{(\text{\rm ad}\overline{u})^k}{(k+1)!}
\]
\[
=y_i+[y_j,\overline{u}]\sum_{k\geq 0}\frac{(\text{\rm ad}\overline{u})^k}{(k+1)!}
+[y_j,u_0]\sum_{k\geq 0}\frac{(\text{\rm ad}\overline{u})^k}{(k+1)!}
\]
\[
=\exp(\text{\rm ad}\overline{u})(y_j)+[y_j,u_0]\sum_{k\geq 0}\frac{(\text{\rm ad}\overline{u})^k}{(k+1)!}.
\]
Easy calculations give
\[
\frac{\partial\exp(\text{\rm ad}\overline{u})(y_j)}{\partial y_i}
=\delta_{ij}+\frac{\partial[y_j,\overline{u}]}{\partial y_i}\sum_{k\geq 0}\frac{t^k}{(k+1)!},
\]
\[
\frac{\partial}{\partial y_i}[y_j,u_0]\sum_{k\geq 0}\frac{(\text{\rm ad}\overline{u})^k}{(k+1)!}
=\frac{\partial [y_j,u_0]}{\partial y_i}\sum_{k\geq 0}\frac{t^k}{(k+1)!},
\]
\[
i,j=1,\ldots,m,\quad t=\sum_{p=1}^mc_pt_p,
\]
where $\delta_{ij}$ is Kronecker symbol.
This gives the expression
\[
J(\exp(\text{\rm ad}u))=I_m+(\overline{D}+D_0)T,
\quad T=\sum_{k\geq 0}\frac{t^k}{(k+1)!},
\]
\[
\overline{D}=\left(\frac{\partial[y_j,\overline{u}]}{\partial y_i}\right),
\quad D_0=\left(\frac{\partial[y_j,u_0]}{\partial y_i}\right).
\]
Since
\[
\frac{\partial[y_j,\overline{u}]}{\partial y_i}=\frac{\partial}{\partial y_i}\left[y_j,\sum_{r=1}^mc_ry_r\right]
=\frac{\partial}{\partial y_i}\sum_{r\neq j}c_r[y_j,y_r]
=\begin{cases}
\sum_{r\neq j}c_rt_r&i=j,\\
-c_it_j&i\neq j,\end{cases}
\]
we obtain the desired form of the matrix $\overline{D}$. Further,
\[
[y_j,u_0]=-\left(\sum_{p>q} [y_p,y_q]h_{pq}(\text{\rm ad}y_q,\ldots,\text{\rm ad}y_m)\right)\text{\rm ad}y_j
\]
\[
\frac{\partial[y_j,u_0]}{\partial y_i}=-t_j\sum_{p>q}\frac{\partial [y_p,y_q]}{\partial y_i}h_{pq}(t_q,\ldots,t_m),
\quad
\frac{\partial [y_p,y_q]}{\partial y_i}
=\begin{cases}
t_q&p=i,\\
 -t_p&q=i,\\
0&p,q\not=i,
\end{cases}
\]
\[
\frac{\partial[y_j,u_0]}{\partial y_i}=-t_jf_i(t_1,\ldots,t_m),
\]
\[
f_i(t_1,\ldots,t_m)=\sum_{q=1}^{i-1}t_qh_{iq}(t_q,\ldots,t_m)-\sum_{p=i+1}^mt_ph_{pi}(t_i,\ldots,t_m).
\]
In this way we obtain the explicit form of the matrix $D_0$.
\end{proof}

Now we shall find the coset representatives of the normal subgroup $\text{\rm Inn}(\widehat{F_m})$
of the group $\text{\rm IA}(\widehat{F_m})$ of IA-automorphisms $\widehat{F_m}$, i.e.,
we shall find a set of IA-automorphisms $\theta$ of $\widehat{F_m}$ such that
the factor group $\text{\rm IOut}(\widehat{F_m})=\text{\rm IA}(\widehat{F_m})/\text{\rm Inn}(\widehat{F_m})$
of the outer IA-automorphisms of $\widehat{F_m}$
is presented as the disjoint union of the cosets
$\text{\rm Inn}(\widehat{F_m})\theta$.

\begin{theorem}\label{iouthatF}
Let $\Theta$ be the set of automorphisms $\theta$ of $\widehat{F_m}$ with Jacobian matrix of the form
\[
J(\theta)=I_m+\left(\begin{array}{llll}
s(t_2,\ldots,t_m)&f_{12}&\cdots&f_{1m}\\
t_1q_2(t_2,t_3,\ldots,t_m)+r_2(t_2,\ldots,t_m)&f_{22}&\cdots&f_{2m}\\
t_1q_3(t_3,\ldots,t_m)+r_3(t_2,\ldots,t_m)&f_{32}&\cdots&f_{3m}\\
\ \ \ \ \ \ \ \vdots&\ \ \vdots&\ \ddots&\ \ \vdots\\
t_1q_m(t_m)+r_m(t_2,\ldots,t_m)&f_{m2}&\cdots&f_{mm}\\
\end{array}\right),
\]
where $s,q_i,r_i,f_{ij}\in K[[t_1,\ldots,t_m]]$ are formal power series without constant terms
and satisfy the conditions
\[
s+\sum_{i=2}^mt_iq_i=0,\quad \sum_{i=2}^mt_ir_i=0,\quad \sum_{i=1}^mt_if_{ij}=0,\quad j=2,\ldots,m,
\]
$r_i=r_i(t_2,\ldots,t_m)$, $i=2,\ldots,m$, does not depend on $t_1$, $q_i(t_i,\ldots,t_m)$,
$i=2,\ldots,m$, does not depend on $t_1,\ldots,t_{i-1}$
and  $f_{12}$ does not contain a summand $dt_2$, $d\in K$.
Then $\Theta$ consists of coset representatives of the subgroup $\text{\rm Inn}(\widehat{F_m})$
of the group $\text{\rm IA}(\widehat{F_m})$ and $\text{\rm IOut}(\widehat{F_m})$ is a disjoint union of
the cosets $\text{\rm Inn}(\widehat{F_m})\theta$, $\theta\in \Theta$.
\end{theorem}

\begin{proof}
Let $A=I_m+(f_{ij})\in I_m+\widehat{S}$,
\[
f_{11}=s,\quad f_{i1}=t_1q_i+r_i,\quad i=2,\ldots,m,
\]
be an $m\times m$ matrix satisfying the conditions of the theorem. The equation
\[
s+\sum_{i=2}^mt_iq_i=0
\]
implies that
\[
t_1s+\sum_{i=2}^mt_i(t_1q_i)=0.
\]
Hence By Lemma \ref{metabelian rule} gives that there exists an $f_1$ in the commutator ideal of $\widehat{F_m}$
such that
\[
\frac{\partial f_1}{\partial y_1}=s,\quad \frac{\partial f_1}{\partial y_i}=t_1q_i,\quad i=2,\ldots,m.
\]
Similarly, the conditions
\[
\sum_{i=2}^mt_ir_i=0,\quad \sum_{i=1}^mt_if_{ij}=0,\quad j=2,\ldots,m,
\]
imply that there exist $f_1',f_j$, $j=2,\ldots,m$, in $\widehat{F_m'}$ with
\[
\frac{\partial f_1'}{\partial y_1}=0,\quad \frac{\partial f_1'}{\partial y_i}=r_i,\quad i=2,\ldots,m,
\]
\[
\frac{\partial f_j}{\partial y_i}=f_{ij},\quad i=1,\ldots,j,\quad j=2,\ldots,m.
\]
This means that $A$ is the Jacobian matrix of a certain IA-automorphism of $\widehat{F_m}$.

Now we shall show that for any $\psi\in \text{\rm IA}(\widehat{F_m})$ there exists an inner automorphism
$\phi=\exp(\text{\rm ad}u)\in\text{\rm Inn}(\widehat{F_m})$ and an automorphism $\theta$ in $\Theta$ such that
$\psi=\exp(\text{\rm ad}u)\cdot \theta$.
Let $\psi$  be an arbitrary element of $\text{\rm IA}(\widehat{F_m})$ and let
\[
\psi(y_1)=y_1+\sum_{k>l} [y_k,y_l]f_{kl}(\text{\rm ad}y_l,\ldots,\text{\rm ad}y_m),
\]
\[
\psi(y_2)=y_2+\sum_{k>l} [y_k,y_l]g_{kl}(\text{\rm ad}y_l,\ldots,\text{\rm ad}y_m),
\]
where $f_{kl}=f_{kl}(t_l,\ldots,t_m),g_{kl}=g_{kl}(t_l,\ldots,t_m)\in K[[t_1,\ldots,t_m]]$.
Let us denote the $m\times 2$ matrix consisting of the first two columns of $J(\psi)$ by
$J(\psi)_2$. Then $J(\psi)_2$ is of the form
\[
J(\psi)_2=\left(\begin{array}{cccc}
1-t_2f_{21}-t_3f_{31}-\cdots-t_mf_{m1}&-t_2g_{21}-t_3g_{31}-\cdots-t_mg_{m1}\\
t_1f_{21}-t_3f_{32}-\cdots-t_mf_{m2}&1+t_1g_{21}-t_3g_{32}-\cdots-t_mg_{m2}\\
t_1f_{31}+t_2f_{32}-\cdots-t_mf_{m3}&\ast\\
\vdots&\vdots\\
t_1f_{m1}+\cdots+t_{(m-1)}f_{m(m-1)}&\ast\\
\end{array}\right),
\]
where we have denoted by $\ast$ the corresponding entries of the second column of
the Jacobian matrix of $\psi$. Let
\[
c_1=-g_{21}(0,\ldots,0),\quad
c_k=f_{k1}(0,\ldots,0),\quad k=2,\ldots,m,
\]
and let us define
\[
\phi_0=\exp(\text{\rm ad}\overline{u}),\quad \overline{u}=\sum_{i=1}^mc_iy_i.
\]
Then
\[
J(\phi_0)=I_m+\left(\begin{array}{llll}
\sum_{i\not=1}c_it_i&-c_1t_2&\cdots&-c_1t_m\\
-c_2t_1&\sum_{i\not=2}c_it_i&\cdots&-c_2t_m\\
-c_3t_1&-c_3t_2&\cdots&-c_3t_m\\
\ \ \ \ \ \ \ \vdots&\ \ \vdots&\ \ddots&\ \ \vdots\\
-c_mt_1&-c_mt_2&\cdots&\sum_{i\not=m}c_it_i\\
\end{array}\right)+B_0,
\]
where the entries of the $m\times m$ matrix $B_0$ are elements of $K[[t_1,\ldots,t_m]]$
which do not contain constant and linear terms. Hence
\[
J(\phi_0\psi)_2=\left(\begin{array}{cccc}
1+t_1s_1(t_1,\ldots,t_m)+s_2(t_2,\ldots,t_m)&g(\widehat{t_2})\\
t_1^2p_2(t_1,\ldots,t_m)+t_1q_2(t_2,\ldots,t_m)+r_2(t_2,\ldots,t_m)&\ast\\
t_1^2p_3(t_1,\ldots,t_m)+t_1q_3(t_2,\ldots,t_m)+r_3(t_2,\ldots,t_m)&\ast\\
\vdots&\vdots\\
t_1^2p_m(t_1,\ldots,t_m)+t_1q_m(t_2,\ldots,t_m)+r_m(t_2,\ldots,t_m)&\ast\\
\end{array}\right),
\]
where $t_1s_1(t_1,\ldots,t_m)$ and $s_2(t_2,\ldots,t_m)$
have no linear terms, and
$g(\widehat{t_2})$ does not contain a summand of the form $dt_2$, $d\in K$.
In the first column of $J(\phi_0\psi)_2$ we have collected the components $t_1^2p_i$ divisible by $t_1^2$,
then the components $t_1q_i$ divisible by $t_1$ only (but not by $t_1^2$)
and finally the components $r_i$ which do not depend on $t_1$, $i=2,\ldots,m$.
By Lemma \ref{metabelian rule} we obtain
\[
t_1^2(s_1+t_2p_2+\cdots+t_mp_m)=0,
\]
\[
t_1(s_2+t_2q_2+\cdots+t_mq_m)=0,
\]
\[
t_2r_2+\cdots+t_mr_m=0.
\]
Let us define $T_s=\{t_s,\ldots,t_m\}$ and rewrite $J(\phi_0\psi)_2$ as
\[
J(\phi_0\psi)_2=\left(\begin{array}{llll}
1-t_1t_2p_2-\cdots-t_1t_mp_m-t_2q_2-\cdots-t_mq_m&g(\widehat{t_2})\\
t_1^2p_2+t_1q_2(T_2)+r_2(T_2)&*\\
t_1^2p_3+t_1q_3(T_2)+r_3(T_2)&*\\
\ \ \ \ \ \ \ \ \ \vdots&\vdots\\
t_1^2p_m+t_1q_m(T_2)+r_m(T_2)&*\\
\end{array}\right),
\]
Now we define
\[
\phi_1=\exp(\text{\rm ad}u_1),\quad u_1=\sum_{i=2}^m[y_i,y_1]p_i(\text{\rm ad}y_1,\ldots,\text{\rm ad}y_m).
\]
The Jacobian matrix of $\phi_1$ has the form
\[
J(\phi_1)=I_m+\left(\begin{array}{llll}
t_1\sum_{i\not=1}t_ip_i&t_2\sum_{i\not=1}t_ip_i&\cdots&t_m\sum_{i\not=1}t_ip_i\\
-t_1^2p_2&-t_1t_2p_2&\cdots&-t_1t_mp_2\\
\ \ \ \ \ \ \ \vdots&\ \ \vdots&\ \ddots&\ \ \vdots\\
-t_1^2p_m&-t_1t_2p_m&\cdots&-t_1t_mp_m\\
\end{array}\right).
\]
The element $u_1$ belongs to the commutator ideal of $\widehat{F_m}$ and
the linear operator $\text{\rm ad}u_1$ acts trivially on $\widehat{F_m'}$.
Hence $\exp(\text{\rm ad}u_1)$ is the identity map restricted on $\widehat{F_m'}$.
Since the automorphism $\phi_0\psi$ is IA, we obtain that
\[
\phi_0\psi(y_j)\equiv y_j\quad (\text{\rm mod }\widehat{F_m'}),
\quad \phi_1(\phi_0\psi(y_j))=\phi_0\psi(y_j)+y_j\text{\rm ad}u_1.
\]
Easy calculations give that
\[
J(\phi_1\phi_0\psi)_2=\left(\begin{array}{cccc}
1-t_2p_2-\cdots-t_mp_m&g(\widehat{t_2})\\
t_1q_2(T_2)+r_2(T_2)&\ast\\
t_1q_3(T_2)+r_3(T_2)&\ast\\
\vdots&\vdots\\
t_1q_m(T_2)+r_m(T_2)&\ast\\
\end{array}\right).
\]
Now we write $q_i(T_2)$ in the form
\[
q_i(T_2)=t_2q_i'(T_2)+q_i''(T_3),\quad i=3,\ldots,m,
\]
and define
\[
\phi_2=\exp(\text{\rm ad}u_2),\quad
u_2=\sum_{i=3}^m[y_i,y_2]q_i'(\text{\rm ad}y_2,\ldots,\text{\rm ad}y_m).
\]
Then we obtain that
\[
J(\phi_2\phi_1\phi_0\psi)_2=\left(\begin{array}{cccc}
1-t_2p_2-\cdots-t_mp_m&g(\widehat{t_2})\\
t_1Q_2(T_2)+r_2(T_2)&\ast\\
t_1q_3''(T_3)+r_3(T_2)&\ast\\
\vdots&\vdots\\
t_1q_m''(T_3)+r_m(T_2)&\ast\\
\end{array}\right),
\]
\[
Q_2(T_2)=q_2(T_2)-\sum_{i=3}^mt_iq_i'(T_2).
\]
Repeating this process we construct inner automorphisms $\phi_3,\ldots,\phi_{m-1}$ such that
\[
\theta=\phi_{m-1}\cdots\phi_2\phi_1\phi_0\psi,
\]
\[
J(\phi_{m-1}\cdots\phi_2\phi_1\phi_0\psi)_2=\left(\begin{array}{cccc}
1+s(T_2)&g(\widehat{t_2})\\
t_1Q_2(T_2)+r_2(T_2)&\ast\\
t_1Q_3(T_3)+r_3(T_2)&\ast\\
\vdots&\vdots\\
t_1Q_m(T_m)+r_m(T_2)&\ast\\
\end{array}\right),
\]
\[
s(T_2)=-t_2p_2(T_2)-\cdots-t_mp_m(T_2).
\]
Hence, starting from an arbitrary coset of IA-automorphisms $\text{\rm Inn}(\widehat{F_m})\psi$,
we found that it contains an automorphism $\theta\in\Theta$ with Jacobian matrix
prescribed in the theorem.
Now, let $\theta_1$ and $\theta_2$ be two different automorphisms in $\Theta$ with
$\text{\rm Inn}(\widehat{F_m})\theta_1=\text{\rm Inn}(\widehat{F_m})\theta_2$. Hence,
there exists a nonzero element $u\in \widehat{F_m}$
such that $\theta_1=\exp(\text{\rm ad}u)\theta_2$. Direct calculations show that this
is in contradiction with the form of $J(\theta_1)$.
\end{proof}

\begin{example}
When $m=2$ the results of Theorems \ref{Jacobian of inner autos} and \ref{iouthatF}
have the following simple form. If $u=\overline{u}+u_0$,
\[
\overline{u}=c_1y_1+c_2y_2,\quad u_0=[y_2,y_1]h(\text{\rm ad}y_1,\text{\rm ad}y_2),
\quad h(t_1,t_2)\in K[[t_1,t_2]],
\]
then the Jacobian matrix of the inner automorphism $\exp(\text{ad}u)$ is
\[
J(\exp(\text{ad}u))=I_2+\left(\begin{matrix}
(c_2+t_1h)t_2&(-c_1+t_2h)t_2\\
-(c_2+t_1h)t_1&-(-c_1+t_2h)t_1\\
\end{matrix}\right)\sum_{k\geq 0}\frac{(c_1t_1+c_2t_2)^k}{(k+1)!}.
\]
The Jacobian matrix of the outer automorphism $\theta\in\Theta$ is
\[
J(\theta)=\left(\begin{matrix}
t_2f_1(t_2)&t_2f_2(t_1,t_2)\\
-t_1f_1(t_2)&-t_1f_2(t_1,t_2)\\
\end{matrix}\right),\quad f_2(0,0)=0.
\]
\end{example}

Recall that the augmentation ideal $\omega$ of the polynomial algebra $K[t_1,\ldots,t_m]$ consists of the polynomials
without constant terms and its completion $\widehat{\omega}$ is the ideal of $K[[t_1,\ldots,t_m]]$ of all formal power
series without constant terms. The elements of the commutator ideal of
the free metabelian nilpotent of class $c$ Lie algebra $L_{m,c}$ are of the form
\[
u_0=\sum_{p>q} [y_p,y_q]h_{pq}(\text{\rm ad}y_q,\ldots,\text{\rm ad}y_m),
\]
where $h_{pq}(t_q,\ldots,t_m)$ belongs to the factor algebra $K[t_1,\ldots,t_m]/\omega^{c-1}$
and may be identified with a polynomial of degree $\leq c-2$. The partial derivative $\partial u/\partial x_i$
belongs to $K[t_1,\ldots,t_m]/\omega^c$ and
may be considered as a polynomial of degree $\leq c-1$. Similarly, the Jacobian matrix of an
endomorphism of $F_{m,c}$ is considered modulo the ideal $M_m(\omega^c)$ of $M_m(K[t_1,\ldots,t_m])$.

As a consequence of our Theorems \ref{Jacobian of inner autos} and \ref{iouthatF}
for $\text{\rm Inn}(\widehat{F_m})$ and $\text{\rm IOut}(\widehat{F_m})$ we
immediately obtain the description of the groups of inner and outer automorphisms of
$L_{m,c}$. We shall give the results for the Jacobian matrices only. The multiplication rule for
the group $\text{\rm Inn}(L_{m,c})$ from Theorem \ref{multiplication in Inn} can be stated similarly.

\begin{corollary}
\text{\rm (i)} Let $u=\overline{u}+u_0\in L_{m,c}$, where
\[
\overline{u}=\sum_{r=1}^mc_ry_r, \quad c_r\in K, r=1,\ldots,m,
\]
is the linear component of $u$ and
\[
u_0=\sum_{p>q} [y_p,y_q]h_{pq}(\text{\rm ad}y_q,\ldots,\text{\rm ad}y_m),
\quad h_{pq}(t_1,\ldots,t_m)\in K[t_1,\ldots,t_m]/\omega^{c-1}.
\]
Then the Jacobian matrix of the inner automorphism $\exp(\text{\rm ad}u)$ is
\[
J(\exp(\text{\rm ad}u))=J(\exp(\text{\rm ad}\overline{u}))+D_0T\equiv I_m+(\overline{D}+D_0)T \quad (\text{\rm mod }M_m(\omega^c)),
\]
\[
\overline{D}=\left(\frac{\partial[y_j,\overline{u}]}{\partial y_i}\right),
\quad D_0\equiv\left(\frac{\partial[y_j,u_0]}{\partial y_i}\right) \quad (\text{\rm mod }M_m(\omega^c)),
\]
\[
T=\sum_{k=0}^{c-1}\frac{t^k}{(k+1)!},\quad t=\sum_{r=1}^mc_rt_r.
\]

\text{\rm (ii)}
The automorphisms with the following Jacobian matrices are
coset representatives of the subgroup $\text{\rm Inn}(L_{m,c})$
of the group $\text{\rm IA}(L_{m,c})$:
\[
J(\theta)=I_m+\left(\begin{array}{llll}
s(t_2,\ldots,t_m)&f_{12}&\cdots&f_{1m}\\
t_1q_2(t_2,t_3,\ldots,t_m)+r_2(t_2,\ldots,t_m)&f_{22}&\cdots&f_{2m}\\
t_1q_3(t_3,\ldots,t_m)+r_3(t_2,\ldots,t_m)&f_{32}&\cdots&f_{3m}\\
\ \ \ \ \ \ \ \vdots&\ \ \vdots&\ \ddots&\ \ \vdots\\
t_1q_m(t_m)+r_m(t_2,\ldots,t_m)&f_{m2}&\cdots&f_{mm}\\
\end{array}\right),
\]
where $s,q_i,r_i,f_{ij}\in \omega/\omega^c$, i.e., are polynomials of degree $\leq c-1$ without constant terms.
They satisfy the conditions
\[
s+\sum_{i=2}^mt_iq_i\equiv 0,\quad \sum_{i=2}^mt_ir_i\equiv 0,\quad \sum_{i=1}^mt_if_{ij}\equiv 0
\quad (\text{\rm mod }\omega^{c+1}), \quad j=2,\ldots,m,
\]
$r_i=r_i(t_2,\ldots,t_m)$, $i=1,\ldots,m$, does not depend on $t_1$, $q_i(t_i,\ldots,t_m)$,
$i=2,\ldots,m$, does not depend on $t_1,\ldots,t_{i-1}$
and  $f_{12}$ does not contain a summand $dt_2$, $d\in K$.
\end{corollary}

\section*{Acknowledgements}

The second named author is grateful to the Institute of Mathematics and Informatics of
the Bulgarian Academy of Sciences for the creative atmosphere and the warm hospitality during his visit when
most of this project was carried out.

\end{document}